\RequirePackage{fix-cm}
\documentclass[smallextended,final,envcountsect]{svjour3}
\smartqed
\usepackage{graphicx}
\usepackage{lscape}
\usepackage{latexsym}
\usepackage{fancybox}
\usepackage{floatrow}
\usepackage{paralist}
\usepackage{lineno}
\usepackage{amsmath,amssymb}
\usepackage{enumerate}
\usepackage{cases}
\usepackage[normalem]{ulem}
\usepackage{color}
\usepackage[colorlinks,linkcolor=blue,anchorcolor=blue,citecolor=blue]{hyperref}
\usepackage{url}
\usepackage[margin=2.5cm,a4paper]{geometry}

\DeclareMathOperator*{\argmin}{arg\,min}

\DeclareMathOperator*{\ri}{ri}

\DeclareMathOperator*{\dom}{dom}
\DeclareMathOperator*{\epi}{epi}

\newtheorem{assumption}{Assumption}
\newcommand{\A}{{\mathcal A}}
\newcommand{\B}{{\mathcal B}}

\newcommand{\cL}{{\mathcal L}}
\newcommand{\M}{{\mathcal M}}
\newcommand{\N}{{\mathcal N}}
\newcommand{\cO}{{\mathcal O}}

\newcommand{\Q}{{\mathcal Q}}

\newcommand{\cS}{{\mathcal S}}
\newcommand{\T}{{\mathcal T}}
\newcommand{\U}{{\mathcal U}}

\newcommand{\X}{{\mathcal X}}
\newcommand{\Y}{{\mathcal Y}}
\newcommand{\Z}{{\mathcal Z}}
\newcommand{\ds}{\displaystyle}

\begin{document}
\journalname{Journal}
\title{A Note on the Convergence of ADMM for Linearly Constrained Convex Optimization Problems
\thanks{The research of the first author was supported by the China Scholarship Council while visiting the National University of Singapore and the National Natural Science Foundation of China (Grant No. 11271117).
The research of the second and the third authors was supported in part by the Ministry of Education, Singapore, Academic Research Fund (Grant No. R-146-000-194-112).}}
\titlerunning{Convergence of ADMM}
\author{Liang Chen\and Defeng Sun\and{Kim-Chuan Toh}}
\institute{Liang Chen
\at College of Mathematics and Econometrics, Hunan University, Changsha, 410082, China.\\
\email{chl@hnu.edu.cn}
\and
Defeng Sun
\at Department of Mathematics and Risk Management Institute, National University of Singapore, 10 Lower Kent Ridge Road, Singapore.\\
\email{matsundf@nus.edu.sg}
\and
Kim-Chuan Toh
\at Department of Mathematics, National University of Singapore, 10 Lower Kent Ridge Road, Singapore.\\
\email{mattohkc@nus.edu.sg}
}

\date{{\large July 8, 2015/} revised on \today}
\maketitle

\begin{abstract}
This note serves two purposes. Firstly, we construct a counterexample to show that the statement on the convergence of the alternating direction method of multipliers (ADMM) for solving  linearly constrained convex optimization problems in a highly influential paper by Boyd et al. [Found. Trends Mach. Learn. 3(1) 1-122 (2011)] can be false if no prior condition  on the existence of solutions to all the subproblems involved is assumed to hold.
Secondly, we present fairly mild conditions to guarantee the existence of solutions to all the subproblems and provide a rigorous convergence analysis on the ADMM, under a more general and useful semi-proximal ADMM (sPADMM) setting considered by Fazel et al. [SIAM J. Matrix Anal. Appl. 34(3) 946-977 (2013)], with a computationally more attractive large step-length that can even exceed the practically much preferred golden ratio of $(1+\sqrt{5})/2$.

\end{abstract}

\section{Introduction}
\label{intro}

Let $\X$, $\Y$ and $\Z$ be three finite-dimensional real Euclidean spaces each endowed with an inner product $\langle\cdot,\cdot\rangle$ and its induced norm $\|\cdot\|$.
Let $f:\Y\to (-\infty,+\infty]$ and $g:\Z\to (-\infty,+\infty]$ be two closed proper convex functions  and
$\A:\X\to\Y$ and $\B:\X\to\Z$ be two linear maps.
Consider the following 2-block separable convex optimization problem:
\begin{equation}
\label{primal}
\min_{y\in\Y,z\in\Z}\big\{ f(y)+g(z) \quad \mbox{s.t.}\quad  \A^*y+\B^*z=c\big\},
\end{equation}
where $c\in\X$ is the given data and the linear maps $\A^*$ and  $\B^*$ are the adjoints of   $\A$ and $\B$, respectively.
The effective domains of $f$ and $g$  are denoted by $\dom\, f$ and $\dom\, g$, respectively.

Let $\sigma>0$ be a given  penalty parameter. The augmented Lagrangian function of problem \eqref{primal} is defined by, for any $(x,y,z)\in\X\times\Y\times\Z$,
\begin{equation}
\label{alagrangian}
\begin{array}{l}
\cL_{\sigma}(y,z;x):=f(y)+g(z)+\langle x ,\A^* y+\B^* z-c\rangle+\frac{\sigma}{2}\|\A^*y+\B^*z-c\|^2\, .
\end{array}
\end{equation}
Choose   an initial point  $(x^0,y^0,z^0)\in\X\times\dom\,  f\times \dom\, g$ and a step-length  $\tau\in (0,+\infty)$. The classical  alternating direction method of multipliers (ADMM) of  Glowinski and Marroco \cite{glo75} and  Gabay and Mercier \cite{Gabay:1976ff} then takes the following scheme for $k=0,1,\ldots$,
\begin{equation}
\label{admm}
\left\{
\begin{array}{l}
\ds y^{k+1}=\argmin_{y}\cL_{\sigma}(y,z^k; x^k),
\\[2mm]
\ds z^{k+1}=\argmin_{z}\cL_{\sigma}(y^{k+1},z; x^k),
\\[2mm]
\ds x^{k+1}=x^k+\tau\sigma(\A^*y^{k+1}+\B^*z^{k+1}-c).
\end{array}
\right.
\end{equation}

The  convergence analysis for the ADMM scheme \eqref{admm} under certain settings was  first conducted  by  Gabay and Mercier \cite{Gabay:1976ff}, Glowinski \cite{glo80book} and Fortin and Glowinski \cite{fortin83}.
One may refer to \cite{boyd11} and  \cite{eck12} for recent surveys on this topic and to \cite{globook14} for a note on the historical development of the ADMM.

In a highly influential paper\footnote{It has been cited $2,229$ times captured by Google Scholar as of July 8, 2015.} written by Boyd et al.  \cite{boyd11}, it was asserted  [Section 3.2.1, Page 17] that if $f$ and $g$ are closed proper convex functions \cite[Assumption 1]{boyd11}
and the Lagrangian function of problem (\ref{primal}) has a saddle point \cite[Assumption 2]{boyd11},  then the ADMM scheme \eqref{admm} converges for $\tau =1$.
This, however, turns to be false without imposing the prior condition that   all the subproblems involved have solutions. To demonstrate our claim,  in this note we shall provide a simple example (see Section \ref{sec:example}) with the following four nice properties:
\begin{itemize}
  \setlength\itemsep{0.1em}
  \item[(P1)] Both $f$ and $g$ are closed proper convex functions;
  \item[(P2)] The Lagrangian function has infinitely many saddle points;
  \item[(P3)] The Slater's constraint qualification (CQ) holds; and
  \item[(P4)] The linear operator $\B$ is nonsingular.
\end{itemize}

Note that our example to be constructed  satisfies  the two assumptions  made in \cite{boyd11}, i.e.,  (P1) and (P2), and the two additional favorable properties (P3) and (P4). Yet,
 the ADMM scheme \eqref{admm} even with $\tau =1$ may  not  be well-defined for solving   problem \eqref{primal}. A closer examination of the proofs given in \cite{boyd11} reveals that the authors mistakenly took for granted the existence of solutions to all the subproblems in (\ref{admm}) under (P1) and (P2) only. Here we will fix this gap by presenting fairly mild   conditions to guarantee the existence of solutions to all the  subproblems in (\ref{admm}).  Moreover, in order to deal with the potentially non-solvability issue of the subproblems in the ADMM scheme \eqref{admm}, we shall analyze  the convergence of the ADMM under a more useful semi-proximal ADMM (sPADMM) setting  advocated by Fazel et al. \cite{fazel13}, with a computationally more attractive large step-length that can even be bigger than  the golden ratio of $(1+\sqrt{5})/2$.

Let $\cS:\Y\to\Y$ and $\T:\Z\to\Z$ be two self-adjoint positive semidefinite linear operators.
Then the sPADMM takes  the following iteration scheme for $k=0,1,\ldots$, \begin{equation}
\label{spadmm}
\left\{
\begin{array}{l}
y^{k+1}=\argmin\limits_{y}\big\{\cL_{\sigma}(y,z^k;x^k)+\frac{1}{2}\|y-y^k\|^2_\cS\big\},
\\[2mm]
z^{k+1}=\argmin\limits_{z}\big\{\cL_{\sigma}(y^{k+1},z;x^k)+\frac{1}{2}\|z-z^k\|^2_\T\big\},
\\[3mm]
x^{k+1}=x^k+\tau\sigma(\A^*y^{k+1}+\B^*z^{k+1}-c).
\end{array}
\right.
\end{equation}
The sPADMM scheme \eqref{spadmm} with $\cS=0$ and $\T=0$ is nothing but the ADMM scheme \eqref{admm} and the case $\cS\succ0$ and $\T\succ 0$ was initiated by Eckstein \cite{eckstein1994}.
Most recent studies have shown that the sPADMM, a seemingly mild extension of the classical ADMM, turns out to play a pivotal role in solving multi-block convex composite conic programming problems \cite{chenl2015,lixd2014,yang2015} with a low to medium accuracy. For more details on choosing $\cS$ and $\T$, one may refer to the recent Ph.D thesis of Li \cite{lithesis}.

The remaining parts of this note are organized as follows. In Section \ref{preliminary}, we first present some necessary preliminary results from convex analysis for later discussions and then provide conditions under which the subproblems  in the sPADMM scheme \eqref{spadmm} are  solvable, or even admit bounded solution sets, so that this scheme is well-defined.
In Section \ref{sec:example}, based on several results established in Section \ref{preliminary}, we construct a counterexample that satisfies (P1)--(P4)  to show that the conclusion on the convergence of ADMM scheme (\ref{admm}) in \cite[Section 3.2.1]{boyd11} can be false without making further assumptions.
In Section \ref{sec:converge}, we establish some satisfactory convergence properties for the sPADMM scheme \eqref{spadmm} with a computationally more attractive large step-length that can  even exceed the golden ratio of  $(1+\sqrt{5})/2$, under fairly weak assumptions. We conclude this note in Section \ref{sec:conclusion}.

\section{Preliminaries}
\label{preliminary}

Let $\U$ be a finite dimensional real Euclidean space endowed with an inner product
 {$\langle\cdot,\cdot\rangle$ and its induced norm  $\|\cdot\|$}. Let  $\cO: \U\to\U$
be any self-adjoint positive semidefinite linear operator. For any $u,u'\in\U$, define $\langle u,u'\rangle_{\cO}:=\langle u,\cO u'\rangle$ and $\|u\|_\cO:=\sqrt{\langle u, \cO u\rangle}$ so that
\begin{equation}
\label{eq:triangle}
\begin{array}{l}
\langle u,u'\rangle_\cO=\frac{1}{2}\left(\|u\|_\cO^2+\|u'\|_\cO^2-\|u-u'\|_\cO^2\right)=\frac{1}{2}\left(\|u+u'\|_\cO^2-\|u\|_\cO^2-\|u'\|_\cO^2\right).
\end{array}
\end{equation}
For any given set $U\subseteq \U$, we denote its relative interior by $\ri(U)$ and define its indicator function $\delta_{U}:\U\to(-\infty,+\infty]$ by
$$
\delta_{U}(u):=\left\{
\begin{array}{ll}
0,&\mbox{if}\quad u\in U,\\
+\infty,\quad &\mbox{if}\quad u\not\in U.
\end{array}
\right.
$$
Let $\theta:\U\to(-\infty,+\infty]$ be a  closed proper convex function.  We use $\dom\,\theta$ and $\epi(\theta)$ to denote its effective domain and its epigraph, respectively.
Moreover, we use $\partial\theta(\cdot)$ to denote the subdifferential mapping \cite[Section 23]{rocbook} of $\theta(\cdot)$, which is defined  by
\begin{equation}
\label{def-sub-d}
\partial \theta(u):=\{v\in\U\,|\, \theta(u')\ge \theta(u)+\langle v,u'-u\rangle\  \forall\, u'\in\U\}, \quad \forall\, u\in\U.
\end{equation}
It holds that there exists a self-adjoint positive semidefinite linear operator $\Sigma_\theta:\U\to\U$ such that for any $u,u'$ with $v\in\partial\theta(u)$ and $v'\in\partial\theta(u')$,
\begin{equation}
\label{monotone}
\langle v-v',u-u'\rangle\ge\|u-u'\|^{2}_{\Sigma_\theta}.
\end{equation}
Since $\theta$ is closed, proper and convex, by \cite[Theorem 8.5]{rocbook} we know that the recession function \cite[Section 8]{rocbook} of $\theta$, denoted by $\theta0^+$, is a positively homogeneous closed proper convex function that can be written as,  for an arbitrary $u'\in\dom\, \theta$,
$$
\theta0^+(u)=\lim_{\rho\to+\infty}\frac{\theta(u'+\rho u)-\theta(u')}{\rho}, \quad \forall\, u\in\U.
$$
The Fenchel conjugate $\theta^{*}(\cdot)$ of $\theta$ is a closed proper   convex function defined by
$$
\theta^*(v):=\sup_{u\in\U}\big\{\langle u,v\rangle-\theta(u)\big\}, \quad \forall\,  v\in\U.
$$
Since $\theta$ is closed, by
\cite[Theorem 23.5]{rocbook} we know that
\begin{equation}
\label{subgradientchange}
v\in\partial\theta(u)
\Leftrightarrow
u\in\partial\theta^{*}(v).
\end{equation}
The dual of problem \eqref{primal} takes the form of
\begin{equation}
\label{dual}
\max_{x\in \X}\big\{h(x):=-f^*(-\A x)-g^*(-\B x)-\langle c,x\rangle\big\}.
\end{equation}
The Lagrangian function of problem \eqref{primal} is defined by
\begin{equation}
\label{lagrangian}
\begin{array}{ll}
\cL(y,z;x):=f(y)+g(z)+\langle x,\A^*y+\B^*z-c\rangle,
\quad\forall\, (y,z,x)\in\Y\times\Z\times\X,
\end{array}
\end{equation}
which is convex in $(y,z)\in \Y\times \Z$ and concave in $x\in \X$.
Recall that  we say the  Slater's  CQ  for problem (\ref{primal}) holds if
\[\bigl\{(y,z)\ |\ y\in\ri(\dom \,f),\ z\in\ri(\dom \,g),\ \A^{*}y+\B^{*}z=c\bigr\}\ne \emptyset.
\]
Under the above Slater's CQ, from \cite[Corollaries 28.2.2 \& 28.3.1]{rocbook} we know that
$(\bar y, \bar z)\in \dom\, f\times \dom\, g$ is a solution to problem \eqref{primal} if and only if there exists a Lagrangian multiplier $\bar x\in\X$ such that
$(\bar x, \bar y, \bar z)$ is a saddle point to the Lagrangian function \eqref{lagrangian}, or,  equivalently, $(\bar x, \bar y,\bar z)$ is a solution to the following Karush-Kuhn-Tucker (KKT) system
\begin{equation}\label{kkt}
-\A x\in\partial f( y),\quad
-\B x\in\partial g( z)\quad
\mbox{and}
\quad
\A^* y+\B^* z=c.
\end{equation}
Furthermore, if the solution set to the KKT system \eqref{kkt} is nonempty, by \cite[Theorem 30.4 \& Corollary 30.5.1]{rocbook} we know that
a vector $(\bar x,\bar y,\bar z)\in\X\times\Y\times\Z$ is a solution to \eqref{kkt} if and only if  $(\bar y, \bar z)$ is an optimal solution to problem \eqref{primal} and  $\bar x$ is an optimal solution to problem  \eqref{dual}.

In the following, we shall conduct  discussions on the existence of solutions to the subproblems in the sPADMM scheme \eqref{spadmm}.  Let the augmented Lagrangian function $\cL_\sigma$ be defined by (\ref{alagrangian}) and $\cS$ and $\T$ be two self-adjoint positive semi-definite linear operators used in the sPADMM scheme (\ref{spadmm}). Let $(x',y',z')\in\X\times\dom\,  f\times \dom\, g$ be an arbitrarily given point.  Consider the following two auxiliary optimization problems:
\begin{equation}
\label{fk}
\begin{array}{l}
\min_{y\in \Y}\big\{F(y):=\cL_{\sigma}(y,z';x')+\frac{1}{2}\|y-y'\|^2_\cS\big\}
\end{array}
\end{equation}
and
\begin{equation}
\label{gk}
\begin{array}{l}
\min_{z\in \Z}\big\{G(z):=\cL_{\sigma}(y',z;x')+\frac{1}{2}\|z-z'\|^2_\T\big\}.
\end{array}
\end{equation}
Note that Since $z'\in\dom\, g$,  problem \eqref{fk} is equivalent to
\begin{equation}
\label{unconstrained}
\begin{array}{l}
\min\limits_{y\in\Y} \big\{
\widehat F(y):=f(y)+\frac{\sigma}{2}\|\A^*y+(\B^*z'-c+{x'}/{\sigma})\|^2+\frac{1}{2}\|y-y'\|^2_\cS\big\}.
\end{array}
\end{equation}
We now study under what conditions    problems \eqref{fk} and \eqref{gk} are solvable or have bounded solution sets. For this purpose,
  we consider  the following  assumptions:
\begin{assumption}
\label{as-1}
$f0^+(y)>0$  for any $y\in\M$, where
$$
\M:=\{y\in\Y\,|\, \A^*y=0,\, \cS y=0\}\backslash\{y\in\Y\, |\, f0^+(-y) = -f0^+(y) = 0\}.
$$
\end{assumption}
\begin{assumption}
\label{as-11}
$g0^+(z)>0$ for any $z\in\N$, where
$$
\N:=\{z\in\Z\,|\, \B^*z=0,\, \T z=0\}\backslash\{z \in \Z\, |\, g0^+(-z)\\ = -g0^+(z) = 0\}.
$$
\end{assumption}
\begin{assumption}
\label{as-2}
$f0^+(y)>0$ for any $0\neq y\in\{y\in \Y\, |\, \A^*y=0, \cS y=0\}$.
\end{assumption}
\begin{assumption}
\label{as-22}
$g0^+(z)>0$ for any $0\neq z\in\{z\in \Z\, |\, \B^*z=0, \T z=0\}$.
\end{assumption}

Note that Assumptions \ref{as-1}-\ref{as-22} are not very restrictive. For example, if both $f$ and $g$ are coercive, in particular if they are norm functions,   all the four assumptions hold automatically without any other conditions. Under the above assumptions,  we have the following results.

\begin{proposition}
\label{theorem:well-defined}
It holds that
\begin{description}
\item[{\rm\bf(a)}]
 Problem $\eqref{fk}$ is solvable if
 Assumption $\ref{as-1}$ holds, and problem $\eqref{gk}$ is solvable if
 Assumption $\ref{as-11}$ holds.
\item[{\rm\bf(b)}]
 The solution set to problem $\eqref{fk}$ is nonempty and  bounded if and only if
Assumption $\ref{as-2}$ holds, and the solution set to problem $\eqref{gk}$ is nonempty and  bounded
if and only if Assumption $\ref{as-22}$ holds.
\end{description}
\end{proposition}
\begin{proof}
{\bf (a)}
We first show that when  Assumption \ref{as-1} holds, the solution set to problem \eqref{fk} is not empty. 
Consider the recession function $\widehat F0^+$ of $\widehat F$.
On the one hand, by using \cite[Theorem 9.3]{rocbook} and the second example given in \cite[Pages 67-68]{rocbook}, we know that for any $y\in\Y$ such that $\A^*y\neq 0$ or $\cS y\neq 0$, one must have  $\widehat F0^+(y)=+\infty$.
On the other hand, for any $y\in\Y$ such that $\A^*y=0$ and $\cS y=0$, by the definition of $\widehat{F}(y)$ in \eqref{unconstrained} we have \[
\widehat F0^+(y)=f0^+(y) +\langle \sigma \A (\B^*z'-c+{x'}/{\sigma})- \cS y^\prime, y\rangle =f0^+(y).
\]
Hence, by Assumption \ref{as-1} we know that $\widehat F0^+(y)>0$ for all $y\in\Y $ except for those satisfying
$\widehat F0^+(-y)=-\widehat F0^+(y)=0$. Then, from  \cite[(b) in Corollary 13.3.4]{rocbook},  it holds that  $0\in\ri(\dom\, \widehat F^*)$. Furthermore, by \cite[Theorem 23.4]{rocbook} we know that $\partial\widehat F^*(0)$ is a nonempty set, i.e., there exists a $\hat y\in\Y$ such that $\hat y\in\partial\widehat F^*(0)$.
By noting that  $\widehat F$ is closed and using    \eqref{subgradientchange}, we then have $0\in\partial\widehat F(\hat y)$,
which implies that $\hat y$ is the solution to problem \eqref{unconstrained} hence to problem \eqref{fk}.

By repeating the above discussions we know that problem \eqref{gk} is also solvable if Assumption \ref{as-11} holds.

\medskip
{\bf (b)} Note that problem \eqref{unconstrained} is equivalent to problem \eqref{fk}. By reorganizing the proofs for part (a),  we can see that
Assumption \ref{as-2} holds if and only if $\widehat {F} 0^+ (y) >0$ for all $0\ne y\in \Y$. As a result, if Assumption \ref{as-2} holds, from \cite[Theorem 27.2]{rocbook} we know that problem \eqref{unconstrained} has a nonempty and bounded solution set. Conversely, if the solution set to problem \eqref{unconstrained} is nonempty and bounded, by \cite[Corollary 8.7.1]{rocbook} we know that there does not exist any $0\neq y\in\Y$ such that  $\widehat F0^{+}(y)\le0$, so that Assumption \ref{as-2} holds.
Similarly, we can prove the remaining results of part (b). This completes the proof of the proposition.
\qed
\end{proof}

Based on Proposition \ref{theorem:well-defined} and its proof, we have the following results.
\begin{corollary}
\label{coro2}
If   problem $\eqref{primal}$ has a nonempty and  bounded solution set, then both problems $\eqref{fk}$ and $\eqref{gk}$ have nonempty and  bounded solution sets.
\end{corollary}
\begin{proof}
Since  problem \eqref{primal} has a nonempty and bounded solution set, there  does not exist any $0\ne y\in\Y$ with $\A^*y=0$ such that $f0^+(y)\le 0$, or $0\ne z\in\Z$ with $\B^*z=0$ such that $g0^+(z)\le 0$.
 Thus, Assumptions  \ref{as-2} and \ref{as-22} hold.
 Then, by part (b) in Proposition \ref{theorem:well-defined} we know that the conclusion of  Corollary \ref{coro2} holds.
\qed
\end{proof}

\begin{proposition}
\label{corpolyhedral}
If $f$ {\rm(}or $g${\rm)} is a closed proper piecewise linear-quadratic convex function \cite[Definition 10.20]{va}, especially a polyhedral convex function,
we can replace the ``\,$>$'' in Assumption $\ref{as-1}\, (\, or\  \ref{as-11}\, )$ by ``\,$\ge$'' and the corresponding sufficient condition in part $($a\,$)$ of Proposition $\ref{theorem:well-defined}$ is also necessary.
\end{proposition}
\begin{proof}
Note that when $f$ is a closed piecewise linear-quadratic convex function, the function $\widehat F$ defined in \eqref{unconstrained} is a piecewise linear-quadratic convex function with $\dom \widehat F= \dom f$ being a closed convex polyhedral set. Then by \cite[Theorem 11.14(b)]{va} we know that $\widehat F^*$ is also a piecewise linear-quadratic convex function whose effective domain is a closed convex polyhedral set. By repeating the discussions for part (a) of  Proposition \ref{theorem:well-defined} and using \cite[Corollary 13.3.4, (a)]{rocbook} we can obtain that
Assumption \ref{as-1}  with ``$>$" being replaced by ``\,$\ge$'' holds if and only if
$0\in\dom\, \widehat F^*$, or $\partial\widehat F^*(0)$ is a nonempty set \cite[Proposition 10.21]{va}, which is equivalent to the fact that $\argmin \widehat F$ is a nonempty set.
If $g$ is piecewise linear-quadratic we can get a similar result.
\qed
\end{proof}

Finally, we need the following easy-to-verify  result on the convergence of quasi-Fej\'er monotone sequences.

\begin{lemma}
\label{lemma:sq-sum}  Let $\{a_k\}_{k\ge0}$ be a nonnegative sequence of real numbers satisfying   $a_{k+1}\le a_k+\varepsilon_k$ for all $k\ge 0$,  where $\{\varepsilon_k\}_{k\ge 0}$ is  a nonnegative and summable sequence of real numbers.
Then the  quasi-Fej\'er monotone sequence $\{a_k\}$ converges to a unique limit point.
\end{lemma}

\section{A Counterexample}
\label{sec:example}
In this section, we shall provide an example that satisfies all the properties (P1)-(P4) stated in Section \ref{intro} to show that the solution set to a certain subproblem in the ADMM scheme \eqref{admm} can be empty if no further assumptions on $f$, $g$ or $\A$  are made.
This means that     the convergence analysis for the  ADMM stated in \cite{boyd11} can be false.
The construction of this example relies  on  Proposition \ref{theorem:well-defined}.
The parameter $\sigma$ and the initial point $(x^{0},y^{0},z^{0})$ in the counterexample are just selected for the convenience of computations and one can construct similar  examples for arbitrary penalty parameters and  initial points.

We now present this example, which is a 3-dimensional 2-block convex optimization problem.

\medskip
\noindent
\fbox{
\begin{minipage}{0.97\textwidth}
\begin{example}
Let $\delta_{\ge 0}(\cdot)$ be the indicator function of the nonnegative real numbers. Consider problem \eqref{primal} with $f(y_{1},y_{2}):=\max(e^{-y_1}+y_2,y_2^2)$, $g(z):=\delta_{\ge 0}(z)$, $\A^{*}=(0,1)$, $\B^{*}=-1$,  and $c=2$, i.e.,
\begin{equation}
\label{problem}
\min_{(y_1,y_2,z)\in\Re^{3}}\Big\{\max(e^{-y_1}+y_2,y_2^2)+\delta_{\ge 0}(z)\ \, |\ \, 0y_{1}+y_2-z=2\Big\}.
\end{equation}
\end{example}
\end{minipage}}

\medskip
In this example,   $f$ and $g$ are closed proper convex functions with $\ri(\dom\, f)=\dom\, f=\Re^2$ and $\ri(\dom\, g)=\{z\, |\, z>0\}\subset \dom\, g $.
The vector $(0,3,1)\in \Re^3$ lies in $\ri(\dom\, f)\times\ri (\dom\, g)$ and satisfies the constraint in problem \eqref{problem}. Hence,  for problem \eqref{problem}, the Slater CQ holds.
It is easy to check  that the optimal solution set to problem \eqref{problem} is given by
$$\{(y_1,y_2,z)\in \Re^3\, |\,  y_1\ge-\log_e 2,\ y_2=2, \ z=0\}$$ and the corresponding optimal objective value is $4$.
The Lagrangian function of problem \eqref{problem} is given by
$$
\cL(y_1,y_2,z;x)=\max(e^{-y_1}+y_2,y_2^2)+\delta_{\ge 0}(z)+x(y_2-z-2), \   \forall\, (y_1,y_2,z,x)\in \Re^4\, .
$$
We now compute the dual of problem \eqref{problem} based on this Lagrangian function.
\begin{lemma}
\label{dualobj}
The objective function of the dual   of  problem $\eqref{problem}$ is given by
$$
h(x)=
\left\{
\begin{array}{ll}
-x^2/4-2x, \quad & \mbox{\rm if}\quad x\in(-\infty,-2),
\\[1mm]
1-x, & \mbox{\rm if}\quad x\in[-2,-1),
\\[1mm]
-2x, &\mbox{\rm if}\quad x\in[-1,0],
\\[1mm]
-\infty, & \mbox{\rm if}\quad x\in(0+\infty).
\end{array}
\right.
$$
\end{lemma}
\begin{proof}
By the definition of the dual objective function,  we have
$$
\begin{array}{l}
h(x)\ds
=\inf_{y_1,y_2,z} \cL(y_1,y_2,z;x)
\\[3mm]
\ds
=\inf_{z\ge0,y_2}\big\{\inf_{y_1}(\max(e^{-y_1}+y_2,y_2^2)+(y_2-z-2)x)\big\}
\\[3mm]

\ds
=\inf_{z\ge 0,y_2}\{\max(y_2,y_2^2)+y_2x-zx-2x\}
\\[3mm]
\ds
=\min_{y_2}
\Big(
\inf_{y_2\in[0,1],z\ge 0}\big\{y_2+y_2x-zx-2x\big\},
\inf_{y_2\not\in[0,1],z\ge 0}\big\{y_2^2+y_2x-zx-2x\big\}\Big).
\end{array}
$$
For any given $x\in\Re$, we have
$$
\begin{array}{l}
\ds
\inf_{y_2\in[0,1],z\ge 0}\big\{y_2+y_2x-zx-2x\big\}
\\[2mm]
\ds
=\inf_{y_2\in[0,1]}\big\{y_2(1+x)\big\}
+\inf_{z\ge 0}\big\{-zx\big\}
-2x
=
\left\{
\begin{array}{ll}
1-x,\quad &\mbox{if}\quad x< -1,
\\[1mm]
-2x, &\mbox{if}\quad x\in[-1,0],
\\[1mm]
-\infty, &\mbox{if}\quad x> 0.
\end{array}
\right.
\end{array}
$$
Moreover, for any $x\in\Re$, it holds that
$$
\begin{array}{l}
\ds
\inf_{y_2\not\in[0,1],z\ge 0}\big\{y_2^2+y_2x-zx-2x\big\}
\\[3mm]
\quad\ds
=\inf_{y_2\not\in[0,1]}\big\{y_2^2+y_2x+x^2/4-x^2/4-2x\big\}+\inf_{z\ge0}\big\{-zx\big\}
\\[3mm]
\quad
\ds
=\inf_{y_2\not\in[0,1]}\big\{(y_2+x/2)^2\big\}+\inf_{z\ge 0}\big\{-zx\big\}-x^2/4-2x
\\[4mm]
\quad
=
\left\{
\begin{array}{ll}
-x^2/4-2x, \quad& \mbox{if}\quad x< -2,
\\[1mm]
1-x, &\mbox{if}\quad x\in[-2,-1],
\\[1mm]
-2x, &\mbox{if}\quad x\in[-1,0],
\\[1mm]
-\infty, &\mbox{if}\quad x> 0.
\end{array}
\right.
\end{array}
$$
Then by combining the above discussions on the two cases we obtain the conclusion of this lemma.
\qed
\end{proof}
\begin{figure}
  \includegraphics[width=0.495\textwidth]{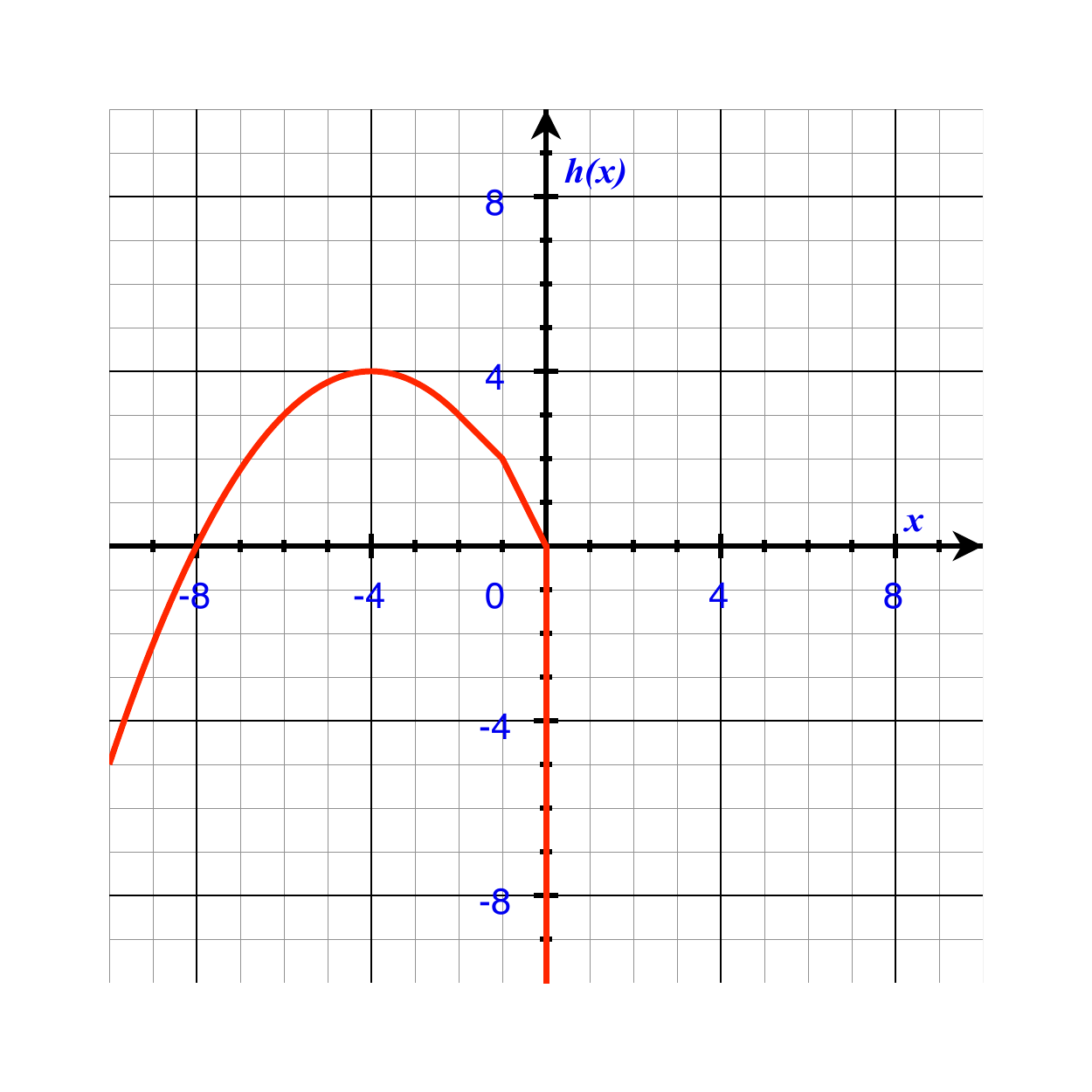}
  \includegraphics[width=0.495\textwidth]{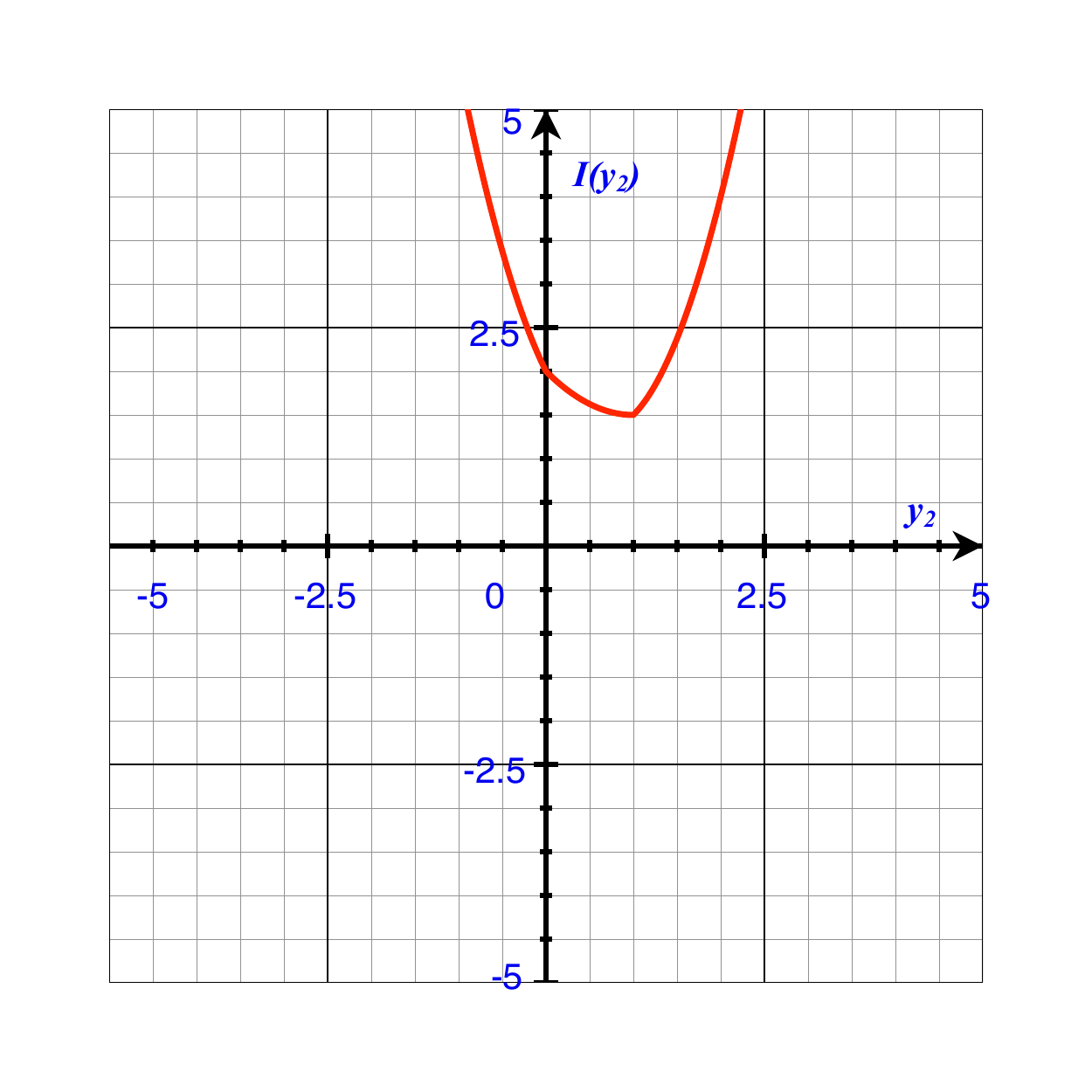}
\caption{Graphs of the dual objective function $h(x)$ (left) and the function $I(y_{2})$ (right).}
\label{fig:1}
\end{figure}

By Lemma \ref{dualobj}, one can  see that the optimal solution to the dual  of problem \eqref{problem} is $\bar x=-4$ and the optimal value of the dual of problem \eqref{problem} is $h(-4)=4$ (see  Fig.  \ref{fig:1}).
Moreover, the set of  solutions to the KKT system \eqref{kkt} for problem \eqref{problem} is given by
$$\big\{(y_1,y_2,z, x)\in \Re^4\ |\  y_1\ge-\log_e 2,\ y_2=2,\ z=0, \ x=-4 \big\}.$$
Next, we consider solving problem \eqref{problem} by using the ADMM scheme \eqref{admm}. For convenience, let $\sigma=1$ and set the initial point $(x^0, y^0_1, y_2^0, z^0)=(0,0,0,0)$.
Now, one should compute $(y_{1}^1,y_{2}^1)$ by solving
$$
\min_{y_1,y_2}
\cL_\sigma (y_1,y_2,z^0;x^0).
$$
Define the function $I(\cdot):\Re\to[-\infty, +\infty]$ by
$$
\begin{array}{ll}
I(y_2):&=\ds
\inf_{y_1}
\cL_\sigma (y_1,y_2,z^0;x^0)
\\[1mm]
&=
\ds
\inf_{y_1}\Big\{
\max\big(e^{-y_1}+y_2,y_2^2\big)+(y_2-2)^2/2
\Big\}
\\[2mm]
&
=\left\{
\begin{array}{ll}
\frac{3}{2} y_2^2-2y_2+2\quad  &\mbox{if}\quad y_{2}\not\in[0,1],
\\[1mm]
\frac{1}{2} y^2_2-y_2+2   &\mbox{if}\quad y_{2}\in[0,1].
\end{array}
\right.
\end{array}
$$
By direct calculations we can see that the above  infimum is attained at $\bar y_2=1$ with $I(\bar y_2)=1.5$ (see Fig. \ref{fig:1}).
However, we have for any $y_1\in\Re$,
$$
\cL_\sigma (y_1,1,0 ;0)
=
\max(e^{-y_1}+1,1)+0.5
=
e^{-y_{1}}+1.5
>\inf_{y_1,y_2}
\cL_\sigma (y_1,y_2,z^0;x^0).
$$
This means that although $\inf_{y_1,y_2}\cL_\sigma(y_1,y_2,z^0; x^0)=1.5$ is finite, it
cannot be attained at any $(y_1,y_2)\in\Re^2$.
Then the subproblem for computing $(y_{1}^{1},y_{2}^{1})$ is not solvable and  hence the   ADMM scheme (\ref{admm})  is not well-defined.
Note that for problem \eqref{problem}, Assumption \ref{as-1} fails to hold  since the direction $y=(1,0)$ satisfies $\A^{*}y=0$ and
$f0^{+}(y)=0$ but $f0^{+}(-y)=+\infty$.

 \begin{remark} The counterexample constructed here is very simple. Yet, one may still ask if the objective function $f$ about $(y_1, y_2)$ in
problem (\ref{problem})  can be replaced by an even simpler quadratic function. Actually, this is not possible  as Assumption \ref{as-1} holds if $f$ is a quadratic function and the original problem has a solution.
Specifically, suppose that $\alpha\in\Re$ is a given number, $\Q:\Y\to\Y$ is a self-adjoint positive semidefinite linear operator and $a\in\Y$ is a given vector while $f$ takes the following form
$$
\begin{array}{l}
f(y)=\frac{1}{2}\langle y,\Q y\rangle+\langle a, y\rangle+\alpha,\quad\forall\, y\in\Y.
\end{array}
$$
From \cite[Pages 67-68]{rocbook} we know that
\begin{equation}
\label{quadratic-recession}
f0^+(y)=\left\{
\begin{array}{ll}
\langle a,y\rangle, &\quad \mbox{if}\quad\Q y=0,\\[1mm]
+\infty,&\quad \mbox{if}\quad\Q y\neq 0.
\end{array}
\right.
\end{equation}
If problem \eqref{primal} has a  solution, one must have $f0^+(y)\ge 0$ whenever $\A^*y=0$.
This, together with \eqref{quadratic-recession}, clearly implies that Assumption \ref{as-1} holds.
\end{remark}

\section{Convergence Properties of sPADMM}
\label{sec:converge}
The example presented in the previous section motivates  us to  consider the convergence of the sPADMM scheme \eqref{spadmm}  with a computationally more attractive large step-length.
We re-emphasize that the sPADMM scheme \eqref{spadmm} is a natural yet more useful extension of the ADMM scheme \eqref{admm} and all the results presented in this section are applicable for the AMMM scheme \eqref{admm}.

For convenience, we introduce some notations, which will be used throughout this section.
We use $\Sigma_{f}$ and $\Sigma_{g}$  to denote the two self-adjoint positive semidefinite linear operators whose definitions, corresponding to the two functions $f$ and $g$ in problem \eqref{primal}, can be drawn from  \eqref{monotone}.
Let $(\bar x,\bar y,\bar z)\in\X\times\Y\times\Z$ be a given vector, whose definition will be specified latter. We denote $x_e:=x-\bar x$, $y_e:=y-\bar y$ and $z_e:=z-\bar z$ for any $(x,y,z)\in\X\times\Y\times\Z$.
If additionally the sPADMM scheme $\eqref{spadmm}$ generates an infinite sequence $\{(x^{k},y^{k},z^{k})\}$, for $k\ge 0$ we denote $x^k_e:=x^k-\bar x$, $y^k_e:=y^k-\bar y$ and $z^k_e:=z^k-\bar z$, and define the following auxiliary notations
\begin{equation}
\label{uv}
\left\{
\begin{array}{ll}
\displaystyle u^{k}:=-\A[x^{k}+(1-\tau)\sigma(\A^*y_e^{k}+\B^*z_e^{k})+\sigma\B^*(z^{k-1}-z^{k})]-\cS(y^{k}-y^{k-1}),
\\[2mm]
\displaystyle v^{k}:=-\B[x^{k}+(1-\tau)\sigma(\A^*y_e^{k}+\B^*z_e^{k})]-\T(z^{k}-z^{k-1}),
\\[2mm]
\Psi_k:=
\frac{1}{\tau\sigma}\|x^k_e\|^2
+\|y_e^{k}\|_\cS^2
+\|z_e^k\|^2_{\T+\sigma\B\B^*},
\color{black}
\\[2mm]
\Phi_k:
=\Psi_{k}+\|z^k-z^{k-1}\|_\T^2
+\max(1-\tau, 1-\tau^{-1})\sigma
\|\A^*y_e^{k}+\B^*z_e^{k}\|^2
\end{array}
\right.
\end{equation}
with the convention $y^{-1}\equiv y^0$ and $z^{-1}\equiv z^0$.
Based on these notations, we have the following result.

\begin{proposition}
\label{prop:2}
Suppose that $(\bar x,\bar y,\bar z)\in\X\times\Y\times\Z$ is a solution to the KKT system $\eqref{kkt}$, and that the sPADMM scheme $\eqref{spadmm}$ generates an infinite sequence $\{(x^{k},y^{k},z^{k})\}$ $($which is guaranteed to be true if  Assumptions $\ref{as-1}$ and $\ref{as-11}$ hold, cf. Proposition $\ref{theorem:well-defined}$$)$.
Then, for any $k\ge 1$,
\begin{equation}
\label{opt_cond}
u^{k}\in\partial f(y^{k}),
\quad
v^{k}\in\partial g(z^{k}),
\end{equation}
\begin{equation}
\label{ineq-last}
\begin{array}{ll}
\Phi_k-\Phi_{k+1}
\ge&
2\|y_e^{k+1}\|_{\Sigma_f}^2
+2\|z_e^{k+1}\|_{\Sigma_g}^2
+\|y^{k+1}-y^k\|_\cS^2
+\|z^{k+1}-z^k\|_\T^2
\\[2mm]
&+\min(1,1-\tau+\tau^{-1})\sigma\|\A^*y_e^{k+1}+\B^*z_e^{k+1}\|^2
\\[2mm]
&+\min(\tau,1+\tau-\tau^2)\sigma\|\B^*(z^{k+1}-z^k)\|^2
\end{array}
\end{equation}
and
\begin{equation}
\label{ineq-large}
\begin{array}{ll}
\Psi_{k}-\Psi_{k+1}
&\ge
2\|y_e^{k+1}\|_{\Sigma_f}^2
+2\|z_e^{k+1}\|_{\Sigma_g}^2
+\|y^{k+1}-y^k\|_\cS^2
+\|z^{k+1}-z^k\|_\T^2
\\[2mm]
&\quad
+(1-\tau)\sigma\|\A^*y^{k+1}_e+\B^*z^{k+1}_e\|^2
+\sigma\|\A^*y^{k+1}_e+\B^*z^{k}_e\|^{2}.
\end{array}
\end{equation}
\end{proposition}
\begin{proof}
For any $k\ge 1$, the inclusions in \eqref{opt_cond} directly follow  from the first-order optimality condition of the subproblems in the sPADMM scheme (\ref{spadmm}).
The inequality \eqref{ineq-last} has been proved in Fazel et al.
\cite[parts (a) and (b) in Theorem B.1]{fazel13}.
Meanwhile, by using (B.12) in \cite[Theorem B.1]{fazel13} and \eqref{eq:triangle} we can get
$$
\begin{array}{l}
\frac{1}{2\tau\sigma}(
\|x^k_e\|^2-\|x^{k+1}_e\|^2)
-\frac{\sigma}{2}\|\B^*(z^{k+1}-z^k)\|^2
-\frac{\sigma}{2}\|\B^*z_e^{k+1}\|^2
+\frac{\sigma}{2}\|\B^*z_e^k\|^2
\\[2mm]
-\frac{2-\tau}{2}\sigma\|\A^*y^{k+1}_e+\B^*z^{k+1}_e\|^2
+\sigma\langle \B^*(z^{k+1}-z^k), \A^*y_e^{k+1}+\B^*z_e^{k+1}\rangle
\\[2mm]
-\frac{1}{2}\|y_e^{k+1}\|_\cS^2
+\frac{1}{2}\|y_e^{k}\|_\cS^2
-\frac{1}{2}\|z_e^{k+1}\|_\T^2
+\frac{1}{2}\|z_e^{k}\|_\T^2
\\[2mm]
\ge \|y_e^{k+1}\|^2_{\Sigma_f}
+\|z_e^{k+1}\|^2_{\Sigma_g}
+\frac{1}{2}\|y^{k+1}-y^k\|_\cS^2
+\frac{1}{2}\|z^{k+1}-z^k\|_\T^2,
\end{array}
$$
which, together with the definition of $\Psi_{k}$ in \eqref{uv}, implies \eqref{ineq-large}. This completes the proof.
\qed
\end{proof}

Now,  we are ready to present several convergence properties of the sPADMM  scheme \eqref{spadmm}.

\begin{theorem}
\label{theorem:c1}
Assume that  the solution set to the KKT system $\eqref{kkt}$ for problem $\eqref{primal}$ is nonempty.
Suppose that  the sPADMM scheme $\eqref{spadmm}$ generates an infinite sequence $\{(x^{k},y^{k},z^{k})\}$, which is guaranteed to be true if  Assumptions $\ref{as-1}$ and $\ref{as-11}$ hold.
Then, if
\begin{equation}
\label{cond:tau}
\tau\in\big(\, 0,(1+\sqrt{5}\,)/2\, \big) \quad\mbox{or}\quad \tau\ge (1+\sqrt{5}\,)/2\ \mbox{but}\  \sum_{k=0}^\infty \|x^{k+1}-x^{k}\|^2<+\infty,
\end{equation}
one has the following results:
\begin{description}
\item [{\rm(a)}]
the sequence $\{x^k\}$ converges to an optimal  solution to the dual problem $\eqref{dual}$, and
the primal objective function value sequence $\{f(y^k)+g(z^k)\}$ converges to the optimal value;

\medskip
\item [{\rm(b)}]
the sequences $\{f(y^{k})\}$ and $\{g(z^{k})\}$ are bounded, and
if  Assumptions $\ref{as-2}$ and $\ref{as-22}$ hold, the sequence $\{y^{k}\}$ and $\{z^{k}\}$ are also bounded;

\medskip
\item[{\rm(c)}]
any accumulation point of the sequence $\{(x^{k},y^{k},z^{k})\}$ is a solution to the KKT system $\eqref{kkt}$, and if $(x^\infty, y^\infty,z^\infty)$ is one of its  accumulation point,
$ \A^{*} y^{k} \to \A^{*} y^{\infty}$,
$ (\Sigma_f +\cS) y^{k} \to (\Sigma_f +\cS) y^{\infty}$,
$ \B^{*} z^{k} \to \B^{*} z^{\infty}$
and
$ (\Sigma_g+\T) z^{k} \to (\Sigma_g +\T) z^{\infty}$ as  $k\to\infty$;

\medskip
\item [{\rm(d)}]
if $\Sigma_f+\A\A^*+\cS\succ 0$ and $\Sigma_g+\B\B^*+\T\succ 0$, then each of the subproblems in the sPADMM scheme $\eqref{spadmm}$ has a unique optimal solution and the whole  sequence $\{(x^{k},y^{k},z^{k})\}$ converges to a solution to the KKT system $\eqref{kkt}$.
\end{description}
\end{theorem}
\begin{proof}
Let $(\bar x,\bar y,\bar z)\in\X\times\Y\times\Z$ be an arbitrary solution to the KKT system $\eqref{kkt}$ of problem $\eqref{primal}$.
We first establish some basic results and then prove (a) to (d) one by one.
In the following, the notations provided at the beginning of this section are used.

Note that $\|\A^{*}y_{e}^{k}\|\le\|\A^{*}y_{e}^{k}+\B^{*}z_{e}^{k}\|+\|\B^{*}z_{e}^{k}\|$ for any $k\ge 0$. Then, if $\tau\in(0,(1+\sqrt{5})/2)$, by using \eqref{uv} and \eqref{ineq-last} we obtain that the sequences
\begin{equation}
\label{sq1}
\{\|x^k\|\},\ \{\|y^{k}\|_{\cS+\sigma\A\A^{*}}\}\  \mbox{and} \  \{\|z^k\|_{\T+\sigma\B\B^*}\}
\end{equation}
are all bounded,
\begin{equation}
\label{sq2}
\sum_{k=0}^{\infty}\|y_e^{k}\|_{\Sigma_f}^2,\,
\sum_{k=0}^{\infty}\|z_e^{k}\|_{\Sigma_g}^2,\,
\sum_{k=0}^{\infty}\|\A^*y_e^{k}+\B^*z_e^{k}\|^2,\,
\sum_{k=0}^{\infty}\|\B^*(z^{k+1}-z^k)\|^2\,
< +\infty
\end{equation}
and
\begin{equation}
\label{sq3}
\sum_{k=0}^{\infty}\|y^{k+1}-y^k\|_\cS^2,\
\sum_{k=0}^{\infty}\|z^{k+1}-z^k\|_\T^2< +\infty.
\end{equation}
If $\tau\ge(1+\sqrt{5})/2$ but $\sum_{k=0}^\infty \|x^{k+1}-x^{k}\|^2<+\infty$, by using the equality that $x^{k+1}-x^{k}=\tau\sigma(\A^*y^{k+1}_e+\B^*z^{k+1}_e)$ we know $\sum_{k=0}^{\infty}\|\A^*y_e^{k}+\B^*z_e^{k}\|^2 < +\infty$.
Therefore, by using $\|\A^{*}y_{e}^{k}\|\le\|\A^{*}y_{e}^{k}+\B^{*}z_{e}^{k}\|+\|\B^{*}z_{e}^{k}\|$
and \eqref{ineq-large} we know that the sequences in \eqref{sq1} are all bounded.
Moreover, it holds that
$$
\|\B^*(z^{k+1}-z^k)\|^2\le 2\|\A^*y^{k+1}_e+\B^*z^{k+1}_e\|^2
+2\|\A^*y^{k+1}_e+\B^*z^{k}_e\|^{2},$$
which, together with \eqref{ineq-large}, implies that \eqref{sq2} and \eqref{sq3} hold.

To sum up, we have shown that when \eqref{cond:tau} holds, the sequences in \eqref{sq1} are bounded and \eqref{sq2} and \eqref{sq3} hold.
This, consequently, implies that $\{u^k\}$ and $\{v^k\}$ are bounded. In the following, we prove (a) to (d) separately.

\medskip
\noindent {\bf (a)}
Since $\{x^k\}$ is a bounded sequence, for any one of its accumulation points, e.g. $x^\infty\in\X$, it admits a subsequence, say, $\{x^{k_{j}}\}_{j\ge 0}$, such that $\lim\limits_{j\to\infty}x^{k_{j}}=x^{\infty}$.
By taking limits in the first two equalities of \eqref{uv} along with $k_j$ for $j\to\infty$ and using \eqref{sq2} and \eqref{sq3}, we obtain that
\begin{equation}
\label{uvlimit}
u^\infty:=\lim_{j\to\infty} u^{k_j}=-\A x^{\infty}
\quad
\mbox{and}
\quad
v^\infty:=\lim_{j\to\infty} v^{k_j} =-\B x^{\infty}
.
\end{equation}
From \eqref{opt_cond} and \eqref{subgradientchange} we know that for any $k\ge 1$,
$
y^{k}\in\partial f^*(u^{k})
$
and
$
z^{k}\in\partial g^*(v^{k})
$. Hence, we can get
$
\A^*y^{k}\in \A^*\partial f^*(u^{k})
$
and
$
\B^*z^{k}\in \B^*\partial g^*(v^{k})
$ so that
\begin{equation}
\label{in:aybz}
\A^*y^{k_j}+\B^*z^{k_j}\in \A^*\partial f^*(u^{k_j})+\B^*\partial g^*(v^{k_j}), \quad\forall\, j\ge 0.
\end{equation}
Then, by using \eqref{sq2}, \eqref{sq3}, \eqref{uvlimit}, \eqref{in:aybz} and the outer semi-continuity of subdifferential mappings of closed proper convex functions we know that
\begin{equation}
\label{opt_dual}
c\in \A^*\partial f^*(-\A x^\infty)+\B^*\partial g^*(-\B x^\infty).
\end{equation}
This implies that $x^\infty$ is a solution to the dual problem \eqref{dual}.
Therefore, we can conclude that any accumulation of $\{x^k\}$ is a solution to the dual problem \eqref{dual}.
To finish the proof of part (a), we need to show that $\{x^k\}$ is a convergent sequence.
This will be done in the following.

We first consider the case that $\tau\in(0,(1+\sqrt{5})/2)$.
Define the sequence $\{\phi_k\}_{k\ge 1}$ by
$$
\phi_k:=
\|y_e^{k}\|_\cS^2
+\|z_e^k\|^2_{\T+\sigma\B\B^*}
+\|z^k-z^{k-1}\|_\T^2
+\max(1-\tau, 1-\tau^{-1})\sigma
\|\A^*y_e^{k}+\B^*z_e^{k}\|^2.
$$
From \eqref{ineq-last} in Proposition \ref{prop:2} and the fact that $\Phi_k\ge\phi_k$, we know that $\{\phi_k\}$ is a nonnegative and bounded sequence. Thus, there exists a subsequence of $\{\phi_{k}\}$, say $\{\phi_{k_{l}}\}$, such that $
\lim\limits_{l\to\infty}\phi_{k_l}=\liminf\limits_{k\to\infty}\phi_k$.
Since $\{x^{k_l}\}$ is bounded, it must has a convergent subsequence, say, $\{x^{k_{l_i}}\}$, such that $\tilde x:=\lim\limits_{i\to\infty}x^{k_{l_i}}$ exists.
Note that $(\tilde x, \bar y,\bar z)$ is a solution to the KKT system \eqref{kkt}.
Therefore, without loss of generality, we can reset $\bar x=\tilde x$ from now on.
By using \eqref{ineq-last} in Proposition \ref{prop:2} we know the nonnegative sequence $\{\Phi_k\}$ is monotonically nonincreasing, and
\begin{equation}\label{phi2}
\lim_{k\to\infty}\Phi_{k}=\lim_{i\to\infty}\Phi_{k_{l_i}}=\lim_{i\to\infty}\big(\frac{1}{\tau\sigma}\|x_e^{k_{l_{i}}}\|^2+\phi_{k_{l_i}}\big)=\liminf_{k\to\infty}\phi_k.
\end{equation}
Since $\frac{1}{\tau\sigma}\|x_e^{k}\|^2=\Phi_k-\phi_k$, we have
\begin{equation}
\label{phi3}
\limsup_{k\to\infty}\frac{1}{\tau\sigma}\|x_e^{k}\|^2=\limsup_{k\to\infty}\{\Phi_k-\phi_k\}
\le
\limsup_{k\to\infty}\, \Phi_k-\liminf_{k\to\infty}\phi_k=0,
\end{equation}
which indicates that $\{x^k\}$ is a convergent sequence.

Second, we need to consider the case that $\tau\ge(1+\sqrt{5})/2$.
Define the nonnegative sequence $\{\psi_k\}$ by
$$
\psi_k:=
\|y_e^{k}\|_\cS^2
+\|z_e^k\|^2_{\T+\sigma\B\B^*},\quad \forall \, k\ge 0.
$$
From \eqref{ineq-large} we known that
$$
\Psi_k-\Psi_{k+1}\ge(1-\tau)\sigma\|\A^*y^{k+1}_e+\B^*z^{k+1}_e\|^2,
$$
which, together with \eqref{sq2}, Lemma \ref{lemma:sq-sum} and the fact that $1-\tau<0$, implies that $\{\Psi_k\}$ is a convergent sequence.
As a result, by the definition of $\psi_k$ we know the sequence $\{\psi_{k}\}$ is nonnegative and bounded.
Then by choosing proper subsequences of $\{\psi_{k}\}$ and $\{x^k\}$ and repeating the previous analysis for getting \eqref{phi2} and \eqref{phi3} with $\phi_k$ and $\Phi_k$ being replaced by $\psi_k$ and $\Psi_k$, we can establish that
$\lim\limits_{k\to\infty}\Psi_{k}=\liminf\limits_{k\to\infty}\psi_k$
and
$\limsup\limits_{k\to\infty}\frac{1}{\tau\sigma}\|x_e^{k}\|^2=0$.
Hence, $\{x^k\}$ is also a convergent sequence.

\medskip
Now we study the convergence of the primal objective function value. One the one hand, since $(\bar x,\bar y,\bar z)$ is a saddle point to the Lagrangian function $\cL(\cdot)$ defined by \eqref{lagrangian}, we have for any $k\ge 1$, $\cL(\bar y,\bar z; \bar x)\le \cL(y^{k},z^{k};\bar x)$. This, together with  $\A^*\bar y+\B^*\bar z=c$, implies that for any $k\ge 1$,
\begin{equation}
\begin{array}{l}
\label{limit1}
f(\bar y)+g(\bar z)-\langle \bar x,\A^*y_e^{k}+\B^*z_e^{k}\rangle
\le f(y^{k})+g(z^{k}).
\end{array}
\end{equation}
On the other hand, from \eqref{opt_cond} and \eqref{def-sub-d} we know that
$$
\begin{array}{l}
f(y^{k})+\langle u^{k}, \bar y-y^{k}\rangle
\le f(\bar y)
\quad
\mbox{and}
\quad
g(z^{k})+\langle v^{k},\bar z-z^{k}\rangle
\le g(\bar z).
\end{array}
$$
By combining the above two inequalities together and using \eqref{uv} we can get
\begin{equation}
\label{limit2}
\begin{array}{l}
f(\bar y)+g(\bar z)
-\langle x^{k},\A^{*}y_e^{k}+\B^{*}z_e^{k}\rangle
-\langle \cS(y^{k}-y^{k-1}), y_e^{k}\rangle
\\[1mm]
-\langle \T(z^{k}-z^{k-1}),z_e^{k}\rangle
-\sigma\langle\B^{*}(z^{k-1}-z^{k}),\A^{*} y_e^{k}\rangle
\\[1mm]
-(1-\tau)\sigma\|\A^{*}y_e^{k}+\B^{*}z_e^{k}\|^2
\ge
f(y^{k})+g(z^{k}).
\end{array}
\end{equation}
Since the sequences in \eqref{sq1} are bounded, by using \eqref{sq2}, \eqref{sq3} and the fact that any nonnegative summable sequence should converge to zero we know the left-hand-sides of both \eqref{limit1} and \eqref{limit2} converge to $f(\bar y)+g(\bar z)$ when $k\to\infty$.
Consequently, $\lim\limits_{k\to\infty}\{f(y^{k})+g(z^{k})\}=f(\bar y)+g(\bar z)$ by the squeeze theorem. Thus, part (a) is proved.

\medskip
\noindent
{\bf (b)}
From \eqref{opt_cond} we konw that for any $k\ge 1$,
\begin{equation}
\label{bd-above}
f(y^{k})
\le  f(\bar y)-\langle u^{k},\bar y-y^{k}\rangle
= f(\bar y)-\langle u^{k},\bar y\rangle+\langle u^{k},y^{k}\rangle.
\end{equation}
On the one hand, from the boundedness of $\{u^k\}$ we know that the sequence $\{-\langle u^{k},\bar y\rangle\}$ is bounded. On the other hand, from \eqref{sq2}, \eqref{sq3} and the boundedness of the sequences in \eqref{sq1}, we can use

$$
\begin{array}{rl}
\langle
u^{k},y^k\rangle=&-\langle
x^{k},\A^*y^k\rangle
-(1-\tau)\sigma\langle\A^*y_e^{k}+\B^*z_e^{k}
,
\A^*y^k\rangle
\\[2mm]
&
-\sigma
\langle
\B^*(z^{k-1}-z^{k}),\A^*y^k\rangle
-
\langle\cS(y^{k}-y^{k-1})
,
y^k
\rangle
\end{array}
$$
to get the boundedness of the sequence $\{\langle u^{k},y^k\rangle\}$.
Hence, from \eqref{bd-above} we know the sequence $\{f(y^{k})\}$ is bounded from above.
From \eqref{kkt} we know
$$
f(y^{k})\ge f(\bar y)+\langle -\A \bar x, y^{k}-\bar y\rangle
=
f(\bar y)-\langle \bar x, \A^*y_e^{k}\rangle.
$$
which, together with the fact that the sequences in \eqref{sq1} are bounded, implies that $\{f(y^{k})\}$ is bounded from below.
Consequently, $\{f(y^{k})\}$ is a bounded sequence. By using similar approach, we can obtain that $\{g(z^{k})\}$ is also a bounded sequence.

Next, we prove the remaining part of (b) by contradiction.
Suppose that Assumption \ref{as-2} holds and the sequence $\{y^{k}\}$ is unbounded.
Note that the sequence $\{y^{k}/(1+\|y^{k}\|)\}$ is always bounded. Thus it must have a subsequence $\{y^{k_{j}}/(1+\|y^{k_{j}}\|)\}_{j\ge 0}$, with $\{\|y^{k_{j}}\|\}$ being unbounded and non-decreasing, converging to a certain point $\xi\in\Y$. From the boundedness of the sequences in \eqref{sq1} we know that $\{\A^{*} y^{k}\}$ and $\{\cS y^{k}\}$ are bounded.
Then we have
$$
\A^{*}\xi
=\A^{*}\Big(
\lim_{j\to\infty}
\frac{y^{k_{j}}}{1+\|y^{k_{j}}\|}
\Big)
=
\lim_{j\to\infty}
\frac{\A^{*}y^{k_{j}}}{1+\|y^{k_{j}}\|}=0.
$$
and, similarly, $\cS\xi=0$.
By noting  that $\|\xi\|=1$, one has  $\xi\in\{y\in\Y\, |\, y\neq 0, \A^*y=0, \cS y=0\}$.
On the other hand, define the sequence $\{d^{k_j}\}_{j\ge 0}$
by
$$
d^{k_j}:=\left(y^{k_{j}}/(1+\|y^{k_{j}}\|)\, ,\, f(y^{k_{j}})/(1+\|y^{k_{j}}\|)\right).
$$
From the boundedness of the sequence $\{f(y^{k_j})\}$ and the definition of $\xi$ we know that $\lim_{j\to\infty}d^{k_j}=(\xi,0)$.
Since $(y^{k_{j}},f(y^{k_{j}}))\in\epi(f)$, by \cite[Theorem 8.2]{rocbook} we know that $(\xi,0)$ is a recession direction of $\epi(f)$. Then from the fact that $\epi(f0^+)=0^+(\epi f)$ we know that $f0^+(\xi)\le 0$, which contradicts Assumption \ref{as-2}.
The boundedness of $\{z^{k}\}$ under Assumption \ref{as-22}  can be similarly proved. Thus, part (b) is proved.

\medskip
\noindent {\bf (c)}
Suppose that $(x^\infty,y^\infty,z^\infty)$ is an accumulation point of $\{(x^k,y^k,z^k)\}$.
Let $\{(x^{k_j},y^{k_j},z^{k_j})\}_{j\ge0}$ be a subsequence of $\{(x^k,y^k,z^k)\}$ which converges to $(x^\infty,y^\infty,z^\infty)$.
By taking limits in \eqref{opt_cond} along with $k_j$ for $j\to\infty$ and using \eqref{uv}, \eqref{sq2} and \eqref{sq3} we can see that
\begin{equation}
-\A x^{\infty}\in\partial f(y^\infty),
\quad
-\B x^{\infty}\in\partial g(z^{\infty})
\quad
\mbox{and}
\quad
\A^*y^\infty+\B^*z^\infty=c,
\end{equation}
which can imply that $(x^\infty,y^\infty,z^\infty)$ is a solution to the KKT system \eqref{kkt}.
Now, without lose of generality we reset $(\bar x,\bar y,\bar z)=(x^\infty,y^\infty,z^\infty)$.
Then, by part (a) we know that  the sequence  $\{\Phi_{k}\}$ defined in \eqref{uv} converges to zero if $\tau\in(0,(1+\sqrt5)/2)$,
and the sequence $\{\Psi_{k}\}$ defined in \eqref{uv} converges to zero if $\tau\ge(1+\sqrt5)/2$ but $\sum_{k=0}^{\infty}\|x^{k+1}-x^{k}\|^{2}<+\infty$.
Thus, we always have
\begin{equation}
\label{unique}
\lim_{k\to\infty}\|y_{e}^{k}\|_{\cS+\Sigma_{f}}=0
\quad
\mbox{and}
\quad
\lim_{k\to\infty}\|z_{e}^{k}\|_{\T+\sigma\B\B^{*}+\Sigma_{g}}=0.
\end{equation}
As a result, it holds that
$\B^{*} z^{k}\to \B^{*} z^{\infty}$,
$(\Sigma_f +\cS) y^{k}\to (\Sigma_f +\cS) y^{\infty}$
and
$(\Sigma_g +\T) z^{k} \to (\Sigma_g + \T) z^{\infty}$ as $k\to\infty$. Moreover, by
using the fact that $\A^{*}y^{k}=(\A^{*}y^{k}+\B^{*}z^{k})-\B^{*}z^{k}$ and $ \A^{*}y^{k}+\B^{*}z^{k} \to \A^{*}y^{\infty}+\B^{*}z^{\infty}=c$ as $k\to \infty$, we can get $ \A^{*} y^{k}  \to \A^{*} y^{\infty}$ as $k\to \infty$. This completes the proof of part (c).

\medskip
\noindent {\bf (d)}
If $\Sigma_f+\cS+\A\A^*\succ 0$ and $\Sigma_g+\T+\B\B^*\succ 0$, the subproblems in the ADMM scheme \eqref{admm} are strongly convex, hence each of them has a unique optimal solution. Then, by part (c) we know that $\{y^k\}$ and  $\{z^k\}$ are convergent. Note that $\{x^k\}$ is convergent by part (a).
Therefore, by part (c) we know that $\{(x^k,y^k,z^k)\}$ converges to a solution to the KKT system \eqref{kkt}.
Hence, part (d) is proved and this completes the proof of the theorem.
\qed
\end{proof}

Before concluding this note, we make the following remarks on the convergence results presented in Theorem
\ref{theorem:c1}.
\begin{remark}
The corresponding results in part (a) of Theorem \ref{theorem:c1} for the ADMM scheme (\ref{admm}) with $\tau=1$ have been stated  in Boyd et al. \cite{boyd11}. However,  as indicated by the counterexample constructed in Section \ref{sec:example},  the proofs in \cite{boyd11} need to be revised with proper  additional assumptions. Actually, no proof on the convergence of $\{x^k\}$ has  been
given in \cite{boyd11} at all. Nevertheless,  one may view the results in part (a) as   extensions of those in  Boyd et al. \cite{boyd11} for the ADMM scheme (\ref{admm}) with $\tau =1$  to a computationally more attractive sPADMM scheme  \eqref{spadmm} with a  rigorous proof.
The condition that $\Sigma_{f}+\A\A^{*}+\cS\succ 0$ and
$\Sigma_{g}+\B\B^{*}+\T\succ 0$ in part (d) was firstly proposed by Fazel et al. \cite{fazel13}.
\end{remark}

\begin{remark}
Note that, numerically, the boundedness of the sequences generated by a certain algorithm is a desirable property and
Assumptions \ref{as-2} and \ref{as-22} can furnish this purpose.
Assumption \ref{as-2}
is pretty mild in the sense that it holds automatically, even if $\cS=0$, for many  practical problems  where $f$ has bounded  level sets. Of course, the same comment can be applied to Assumption \ref{as-22}. \end{remark}

\begin{remark}
The sufficient condition that $\tau\ge(1+\sqrt{5})/2$ but $\sum_{k=1}^\infty\|x^{k+1}-x^{k}\|^2<+\infty$ simplifies
the condition proposed by Sun et al.\footnote{
The condition that $\tau\ge(1+\sqrt{5})/2$ but $\sum_{k=1}^\infty\{\|\B^*(z^{k+1}-z^k)\|^2+\sigma\|x^{k+1}-x^{k}\|^{2}\}<+\infty$ was used in \cite[Theorem 2.2]{yang2015}.
} \cite{yang2015}
for the purpose of achieving  better numerical performance.
The advantage of taking the step-length $\tau\ge(1+\sqrt{5})/2$ has been observed in \cite{chenl2015,lixd2014,yang2015} for solving high-dimensional linear and convex quadratic semi-definite programming problems. In numerical computations, one can start with a larger $\tau$, e.g. $\tau=1.95$, and reset it as $\tau:=\max(\gamma\tau,1.618)$ for some $\gamma\in(0,1)$, e.g. $\gamma =0.95$,  if at the $k$-th iteration one observes that
$\|x^{k+1}-x^{k}\|^2>c_0/k^{1.2}$ for some given positive constant $c_0>0$. Since $\tau$ can be reset at most a finite number of times, our convergence analysis is valid for such a strategy.
One may refer to \cite[Remark 2.3]{yang2015} for more discussions on this computational issue.
\end{remark}

\section{Conclusions}
\label{sec:conclusion}
In this note, we have constructed a  simple example possessing several nice properties    to illustrate  that      the convergence theorem of    the ADMM scheme (\ref{admm})  stated in  Boyd et al. \cite{boyd11} can be false if no prior condition that guarantees the existence of solutions to all the subproblems involved is made.  In order to correct this mistake we have presented fairly mild conditions under which all the subproblems are solvable by using standard knowledge in convex analysis. Based on these conditions, we have further  conducted  the convergence analysis of the ADMM under a more general and useful sPADMM setting, which has the  the flexibility of allowing  the users to choose proper proximal terms to guarantee the existence of solutions to  the subproblems. In particular, we have established some satisfactory convergence properties of the sPADMM with a computationally more attractive large step-length that can exceed the golden ratio of 1.618. In conclusion,  this note has (i) clarified some confusions on the convergence results of the popular ADMM; (ii) opened the potential for designing computationally more efficient ADMM-type solvers in the future.

\end{document}